\documentclass[10pt, a4paper]{amsproc}
\usepackage{amsmath,amssymb,amsthm,enumerate,lineno,mathtools}
\usepackage{afterpage,diagbox,multirow,pdflscape,threeparttable} 
\usepackage{rotating}
\usepackage{hyperref}
\usepackage{footnote}
\usepackage{hhline}

\newtheorem{theorem}{Theorem}[section]
\newtheorem{proposition}[theorem]{Proposition}
\newtheorem{lemma}[theorem]{Lemma}
\newtheorem{corollary}[theorem]{Corollary}

\theoremstyle{definition}
\newtheorem{definition}[theorem]{Definition}

\theoremstyle{remark}
\newtheorem{remark}[theorem]{Remark}

\newcommand{\ch}{\text{\rm{ch}}}
\newcommand{\cvOT}{c_v^{\text{\rm{OT}}}}
\newcommand{\divisible}{\mathrm{div}}
\newcommand{\eK}{e_K}
\newcommand{\et}{\text{\rm{\'{e}t}}}
\newcommand{\Gal}{\mathrm{Gal}}
\newcommand{\Gm}{\mathbb{G}_\mathrm{m}}
\newcommand{\kab}{k^{\mathrm{ab}}}
\newcommand{\kbar}{\overline{k}}
\newcommand{\Kgm}{K_{g,\m}}
\newcommand{\kgm}{k_{g,\m}}
\newcommand{\m}{\boldsymbol{m}}
\newcommand{\mK}{{\mathfrak m}_K}
\newcommand{\OK}{{\mathcal O}_K}
\newcommand{\phiLK}{\varphi_{L/K}}
\newcommand{\phiLKp}{\varphi_{L/K'}}
\newcommand{\psiLK}{\psi_{L/K}}
\newcommand{\psiKpK}{\psi_{K'/K}}
\newcommand{\Q}{\mathbb{Q}}
\newcommand{\Ql}{\Q_\ell}
\newcommand{\Qp}{\Q_p}
\newcommand{\R}{\mathbb{R}}
\newcommand{\rank}{\mathrm{rank}}
\newcommand{\separable}{\mathrm{sep}}
\newcommand{\tor}{\mathrm{tor}}
\newcommand{\UKu}{U_K^u}
\newcommand{\Z}{\mathbb{Z}}

\begin{document}

\title[Mordell--Weil groups over large algebraic extensions]{Mordell--Weil groups over large algebraic extensions of fields of characteristic zero}

\author{Takuya Asayama}
\address{Center for Innovative Teaching and Learning, Institute of Science Tokyo, 2-12-1 Ookayama, Meguro-ku, Tokyo 152-8550, Japan}
\email{asayama@citl.isct.ac.jp}

\author{Yuichiro Taguchi}
\address{Department of Mathematics, Institute of Science Tokyo, 2-12-1 Ookayama, Meguro-ku, Tokyo 152-8551, Japan}
\email{taguchi@math.titech.ac.jp}


\date{}

\begin{abstract}
We study the structure of the Mordell--Weil groups of 
semiabelian varieties over
large algebraic extensions 
of a finitely generated field of characteristic zero. 
We consider two types of algebraic extensions in this paper; one is of extensions obtained by adjoining the coordinates of certain points of various semiabelian varieties; the other is of extensions obtained as the fixed subfield in an algebraically closed field by a finite number of automorphisms. 
Some of such fields turn out to be new examples of 
Kummer-faithful fields which are not sub-$p$-adic. 
Among them, we find both examples of Kummer-faithful fields 
over which the Mordell--Weil group modulo torsion 
can be free of infinite rank and not free.
\end{abstract}

\maketitle

\section{Introduction}
\label{sec-introduction}

In this paper,  we study the structure of the Mordell--Weil groups of
semiabelian varieties over large algebraic extensions of
a finitely generated field of characteristic zero.
In~1974, Frey and Jarden~\cite[p.~127]{FreyJ} conjectured that,
if $A$ is an abelian variety over an algebraic number field $k$,
then the Mordell--Weil group $A(\kab)$ of $A$ over the
the maximal abelian extension
$\kab$ of $k$ is of infinite rank.
There have been many works related to this conjecture (\cite{RW, Im06REC, Kobayashi, Petersen, SY, IL13ir}, to cite a few).
Some authors~\cite{Moon09, Moon12, GHP} have proved that,
for certain algebraic extensions $K/k$ of infinite degree,
the Mordell--Weil group $A(K)$ is not only of infinite rank
but also free modulo torsion. 
In particular, $A(K)$ modulo torsion has trivial divisible part; 
it is this property that we focus on in this paper. 
A perfect field $K$ is called \textit{Kummer-faithful} if, for any 
semiabelian variety $A$ over a finite extension $L$ of $K$, 
the Mordell--Weil group $A(L)$ has trivial divisible part 
(for more detail, see Definition~\ref{def:kummer faithful}).
In this paper, we obtain a wide class of Kummer-faithful fields that are not \textit{sub-$p$-adic} (see Definition~\ref{def:kummer faithful} for the definition).
This improves some of the previous results by Ozeki--Taguchi~\cite{OT} and by Ohtani~\cite{Ohtani22} (see also its corrigendum~\cite{Ohtani23}).
The notion of Kummer-faithfulness (resp.\ sub-$p$-adicness) originates from anabelian geometry and it is expected (resp.\ known) that a field possessing this property makes it suitable as a ground field for anabelian geometry.

Let  $k$  be a finite extension of the rational number field  $\Q$. 
As in~\cite{Ohtani22, Ohtani23}, we consider two types of algebraic extensions of  $k$. 
One is of extensions of  $k$  obtained by adjoining the coordinates of 
certain points of semiabelian varieties over  $k$. 
The other is of extensions obtained as the fixed subfields of 
an algebraic closure  $\kbar$  by some automorphisms
(in the study of the latter type of extensions, 
we allow  $k$  to be 
not only finite extensions of  $\Q$ but also 
arbitrary finitely generated fields of characteristic zero). 
Let us recall some key definitions for us:

\begin{definition}\label{def:kummer faithful}
\begin{enumerate}[(\roman{enumi})]
    \item (\cite[Definition~1.5]{Mochizuki15})
    A perfect field  $k$  is said to be \textit{Kummer-faithful} (resp.\ \textit{torally Kummer-faithful})
    if, for every finite extension $L$ of $k$ and every semiabelian variety (resp.\ every torus)
    $A$ over $L$, the divisible part
    $A(L)_\divisible=\bigcap_{n = 1}^{\infty} nA(L)$  of the Mordell--Weil group 
    $A(L)$  of  $A$  over  $L$  is trivial; 
    \[
    A(L)_\divisible\ =\ 0.
    \]
    \item (\cite[Definition~2.6~(2)]{OT})
    A perfect field  $k$  is said to be \textit{highly Kummer-faithful} 
    if, for every finite extension $L$ of $k$ and 
    every proper smooth 
    variety $X$ over $L$, it holds that 
    \[
    H^i_\et(X_{\kbar},\Ql(r))_{G_L}=0
    \]
    for any prime number  $\ell\not=\ch(k)$  and  
    any  $i, r$  with  $i\not=2r$, where the suffix 
    $(\text{--})_{G_L}$ means the largest quotient of 
    $(\text{--})$  on which the absolute Galois group $G_L$ of $L$ acts trivially.
    \item (\cite[Definition~15.4~(i)]{Mochizuki99})
    A field  $k$  is said to be \textit{sub-$p$-adic} 
    if there exists a prime number $p$  and 
    a finitely generated field extension $L$ of the $p$-adic number field $\Qp$ such that 
    $k$ is isomorphic to a subfield of $L$. 
\end{enumerate}
\end{definition}

Recall also that
the Grothendieck conjecture has been proved for 
hyperbolic curves over sub-$p$-adic fields~\cite{Mochizuki99} and, later, 
in~\cite{Mochizuki15}, \cite{Hoshi} etc., 
Mochizuki and Hoshi try to extend such results to 
hyperbolic curves over Kummer-faithful fields. 
Thus it would be interesting to supply many examples 
of Kummer-faithful fields that are not sub-$p$-adic.

It is known that:
\begin{itemize}
    \item if a \textit{Galois} extension $k$ of a Kummer-faithful field of characteristic zero is highly Kummer-faithful, then $k$ is Kummer-faithful~\cite[Proposition~2.8]{OT};
    \item a sub-$p$-adic field is Kummer-faithful, 
    but the converse is not true~\cite[Remark~1.5.4, (i) and (iii)]{Mochizuki15}.
\end{itemize}


To explain our \textit{first type} of results, 
let  $k$  be a finite extension of  $\Q$. 
For an integer  $g\geq 1$  and a family 
$\m=(m_p)_{p:\mathrm{prime}}$ of integers 
$m_p\geq 0$ indexed by prime numbers, let
\[
\kgm \coloneqq k(A(\kbar)[p^{m_p}] \mid A, p)
\] 
be the extension field of $k$ obtained by adjoining all coordinates of 
the $p^{m_p}$-torsion points of  $A$  
for 
all semiabelian variety  $A$  over  $k$  of dimension $\leq g$  and 
all prime numbers  $p$. 
In~\cite[Theorem~3.3]{OT}, it is proved that 
$\kgm$  is highly Kummer-faithful. 
In this paper, we consider the following larger field; let
\[
\Kgm \coloneqq k(p^{-m_p}A(k) \mid A, p)
\] 
be the extension field of  $k$ 
obtained by adjoining all coordinates of the points in 
\[
p^{-m_p}A(k) \coloneqq 
\{P\in A(\kbar)\mid p^{m_p}P\in A(k)\}
\] 
for 
all semiabelian variety  $A$  over  $k$  of dimension $\leq g$  and 
all prime numbers  $p$. 
Note that  $\Kgm$  contains  $\kgm$  as a subfield.
Our first main result is:

\begin{theorem}[$=$~Theorem~\ref{MT1-sec2}]\label{MT1}
The field  $\Kgm$  is highly Kummer-faithful.
\end{theorem}

This extends Theorem 3.3 of \cite{OT}, which proved 
the same for the field $\kgm$, to the case of 
much larger field $\Kgm$.
Note that, if $m_p\geq 1$ for infinitely many primes $p$, 
then the field  $\Kgm$  is \textit{not} sub-$p$-adic, 
since the smaller field  
$\kgm$  is already not sub-$p$-adic~\cite[Remark~3.4]{OT}.

Another interesting feature of the field  $\Kgm$  is the following:


\begin{proposition}[$=$~Proposition~\ref{prop:not free-sec2}]\label{prop:not free}
Let  $A$  be a semiabelian variety  $A$  over  $k$  of dimension $\leq g$. 
If  $A(k)$  has a non-torsion point (i.e., if  $\rank(A(k))\geq 1$)  and  
$m_p\geq 1$  for infinitely many primes  $p$, then
the Mordell--Weil group  $A(\Kgm)$  modulo the 
torsion subgroup  $A(\Kgm)_\tor$  is {\rm not} free. 
In particular, it is not finitely generated. 
\end{proposition}

In some sense, Theorem~\ref{MT1} and Proposition~\ref{prop:not free} 
refer to \textit{opposite} characters of the field  $\Kgm$, because 
being Kummer-faithful means that the field is ``not too large'', whereas 
the fact that the Mordell--Weil group modulo torsion is not free may 
be interpreted as saying that the field is ``not too small". 
We discuss the relations between the size of fields and the Mordell--Weil groups over it in Remark~\ref{rem:R1}.


The \textit{second type} of extensions that we treat in this paper is of fields close to an algebraically closed field, meaning that they are obtained by cutting out of an algebraic closure of $K$ as the fixed subfield by finitely many automorphisms $\sigma_1, \ldots, \sigma_e$.
Fix an algebraic closure $\overline{K}$ of any field $K$ and let $K^\separable$ be the separable closure of $K$ in $\overline{K}$.
Let $G_K$ be the absolute Galois group $\Gal(K^\separable / K)$ of $K$ and $e$ a positive integer.
For any $\sigma = (\sigma_1, \ldots, \sigma_e) \in G_K^e$ (the direct product of $e$ copies of $G_K$), set $\overline{K}(\sigma)$ to be the fixed field of $\sigma$ in $\overline{K}$.
We also set $\overline{K}[\sigma]$ to be the maximal Galois subextension of $K$ in $\overline{K}(\sigma)$.
We equip the compact group $G_K^e$ with the normalized Haar measure $\mu = \mu_{G_K^e}$, which allows $G_K^e$ to be regarded as a probability space~\cite[Section~21.1]{FriedJ}.
The term \textit{almost all} $\sigma \in G_K^e$ is used in the sense of ``all $\sigma \in G_K^e$ outside some measure zero set".

Several studies on the structure of the Mordell--Weil groups of semiabelian varieties over $\overline{K}(\sigma)$ and over $\overline{K}[\sigma]$ have been carried out, which are summarized in Theorem~\ref{knownresultforKbarsigma} for the convenience of reference in later discussions.
We describe more detailed structures of the Mordell--Weil groups of semiabelian varieties over finite extensions of $\overline{K}(\sigma)$ and of $\overline{K}[\sigma]$ when $K$ is a finitely generated field over $\mathbb{Q}$.
Our first result in this context is the following:

\begin{theorem}[$=$ Theorem~{\ref{mythm-FreenessOfMWGmodTor}}]\label{mythm-Freeness-intro}
Suppose that $K$ is a finitely generated field over $\mathbb{Q}$ and $e \ge 2$.
Then, for almost all $\sigma \in G_K^e$, the following statement holds:
for any finite extension $L$ of $\overline{K}[\sigma]$ and any semiabelian variety $A$ of positive dimension over $L$, the group $A(L) / A(L)_\tor$ is a free $\mathbb{Z}$-module of rank $\aleph_0$.
\end{theorem}

We also investigate the Kummer-faithfulness of $\overline{K}(\sigma)$ and of $\overline{K}[\sigma]$.
Ohtani showed that, if $K$ is a number field and $e \ge 2$, then any finite extension of $\overline{K}[\sigma]$ is Kummer-faithful for almost all $\sigma \in G_K^e$ (see~\cite[Corollary~1]{Ohtani22} and its corrigendum~\cite[Corollary~1]{Ohtani23}).
Our next result contains an improvement over the one by Ohtani, that is, we show that the same conclusion holds even if $e = 1$ for arbitrary finitely generated field $K$ over $\mathbb{Q}$.
For the field $\overline{K}(\sigma)$, we obtain a partial result on its multiplicative group for any $e \ge 1$.
We remark that neither $\overline{K}(\sigma)$ nor $\overline{K}[\sigma]$ are sub-$p$-adic for almost all $\sigma \in G_K^e$ and any prime number $p$ (see Proposition~\ref{prop-NotSubpAdic}).

\begin{theorem}\label{mythm-SummaryOfKFnessForKbarsigma}
Let $K$ be a finitely generated field over $\mathbb{Q}$ and $e$ a positive integer.
\begin{enumerate}[\textup{(\arabic{enumi})}]
    \item \textup{($=$~Theorem~\ref{mythm-KFforKbarbrasigma})} Any finite extension of $\overline{K}[\sigma]$ is Kummer-faithful for almost all $\sigma \in G_K^e$.
    \item \textup{($=$~Theorem~\ref{mythm-TKF})} The multiplicative group of $\overline{K}(\sigma)$ has no nontrivial divisible point for almost all $\sigma \in G_K^e$.
\end{enumerate}
\end{theorem}
    

\afterpage{
\begin{landscape}
\begin{table}
    \caption{Possible structure of Mordell--Weil groups of abelian varieties $A$ and examples of $K$}
    \label{table:PossibleStructure}
    \centering
    \begin{threeparttable}
    \centering
    \begin{tabular}{|c|c|c|c|c|c|c|c|c|}
        \hline
        \multicolumn{4}{|c|}{\multirow{3}{*}{\diagbox[width=5.9cm,height=3\line]{$A(K) / A(K)_\tor$}{$A(K)_\tor$}}} & \multicolumn{3}{c|}{infinite} & \multirow{3}{*}{finite} \\ \cline{5-7}
        \multicolumn{4}{|c|}{} & \multicolumn{1}{c}{} & \multirow{2}{*}{$\begin{array}{c}
            \text{has non-trivial} \\
            \text{divisible part}
        \end{array}$} & \multirow{2}{*}{$\begin{array}{c}
            \text{has trivial} \\
            \text{divisible part}
        \end{array}$} & \\ \cline{5-5}
        \multicolumn{4}{|c|}{} & \multicolumn{1}{c|}{$(\Q / \Z)^{\oplus 2 g}$} & & & \\ \hline
        \multirow{4}{*}{$\begin{array}{c}
            \text{of} \\
            \text{inf.} \\
            \text{rank}
        \end{array}$} & \multirow{3}{*}{$\begin{array}{c}
            \text{not} \\
            \text{free}
        \end{array}$} & \multicolumn{1}{c}{} & \multicolumn{1}{|c|}{$\begin{array}{c}
            \text{inf.\ dim'l} \\
            \Q\text{-vect.\ sp.}
        \end{array}$} & \multicolumn{1}{c|}{$\begin{array}{c}
            \bullet\, K\text{: alg.\ clsd.} \\
            \text{(classical)}
        \end{array}$} & $\begin{array}{c}
            \bullet\, K\text{: real clsd.} \\
            \text{\cite{Lowry}}
        \end{array}$ & & \\ \cline{4-8}
        & & \multicolumn{2}{c|}{$\begin{array}{c}
            \text{has non-trivial} \\
            \text{divisible part}
        \end{array}$} & \multicolumn{1}{c|}{} & & $\diamond\, K = \kbar(\sigma),\, \sigma \in G_k \tnote{a}$ & \\ \cline{3-8}
        & & \multicolumn{2}{c|}{$\begin{array}{c}
            \text{has trivial} \\
            \text{divisible part}
        \end{array}$} & \multicolumn{1}{c|}{} & & $\begin{array}{c}
            \bullet\, K = \Kgm \tnote{b} \\
            \diamond\, K = \kbar(\sigma),\, \sigma \in G_k \tnote{a}
        \end{array}$ & \\ \cline{2-8}
        & \multicolumn{3}{c|}{free} & \multicolumn{1}{c|}{$\begin{array}{c}
            \bullet\, K = k(A(\kbar)_\tor) \tnote{c} \\
            \text{\cite{Larsen05}}
        \end{array}$} & & $\diamond\, K = \kbar(\sigma),\, \sigma \in G_k \tnote{a}$ & $\begin{array}{c}
            \bullet\, K = \kbar[\sigma] \text{ for a.a.} \\
            \sigma \in G_k^e \text{ with } e \ge 2 \tnote{d}
        \end{array}$ \\ \hline
        \multicolumn{4}{|c|}{of finite rank} & \multicolumn{1}{c|}{} & & & $\begin{array}{c}
            \bullet\, K\text{: fin.\ ext.\ of } k \\
            \text{(Mordell--Weil)}
        \end{array}$ \\ \hline
    \end{tabular}
    \begin{tablenotes}
          \item[a]{We know that the field $\kbar(\sigma)$ is located in one of the three cells shown in the table for almost all $\sigma \in G_k$ (by~\cite{FreyJ, Zywina16, JP19}), but do not know which one it is in.}
          \item[b]{For $\m = (m_p)_p$ with $m_p \ge 1$ for infinitely many $p$ and $A$ of dimension $\le g$ with $\rank(A(k)) \ge 1$ (by Theorem~\ref{MT1} ($=$~Theorem~\ref{MT1-sec2}) and Proposition~\ref{prop:not free} ($=$~Proposition~\ref{prop:not free-sec2})).}
          \item[c]{Larsen~\cite{Larsen05} proved that $A(K) \cong M \oplus (\Q / \Z)^{\oplus 2 g}$ for any abelian variety $A$ over $k$ of dimension $g$, where $K = k(A(\kbar)_\tor)$ and $M$ is a free $\Z$-module, and proved that $\rank(M) = \infty$ if $A$ is an elliptic curve (the result for the case where $k = \mathbb{Q}$ and $A$ is an elliptic curve was independently obtained by Habegger~\cite{Habegger}).
          He asked whether $\rank(M)$ is infinite for a general $A$, but it remains open.}
          \item[d]{Follows by Theorem~\ref{mythm-Freeness-intro} ($=$~Theorem~\ref{mythm-FreenessOfMWGmodTor}) and~\cite{JJ01, GJ06}.}
    \end{tablenotes}
    \end{threeparttable}
\end{table}
\end{landscape}
}


\begin{remark}\label{rem:R1}

    We take this opportunity to summarize the situation of our interest in
    Table~\ref{table:PossibleStructure}:
    Let  $A$  be an abelian variety 
    (not semiabelian, for simplicity) 
    of dimension  $g\geq 1$  
    defined over a finitely generated field  
    $k$  of characteristic zero, and 
    $K$ an algebraic extension of  $k$. 
    We denote by 
    $A(K)_\tor$  the torsion subgroup of the Mordell--Weil group  $A(K)$  
    of  $A$  over  $K$.
    Among the fields appearing in the table, Kummer-faithful fields are $\Kgm$ for any $g$ and any $\m$, $\kbar[\sigma]$ for almost all $\sigma \in G_k^e$ with $e \ge 2$ (actually with $e \ge 1$; the freeness result is not known for $e = 1$), and any finite extension $K$ of $k$.
    The algebraic closure and the real closure of $k$ are not Kummer-faithful.
    Recently, Lowry~\cite[Theorem~1]{Lowry} determined the structure of $A(K)$ when $K$ is a real closed field.
    Another example of a non-Kummer-faithful field $K$ is $K = k(A(\kbar)_\tor)$ for a fixed $A$; obviously $A(K)$ contains a non-trivial divisible subgroup $A(\kbar)_\tor$.
    Larsen~\cite[Theorem~3.1]{Larsen05}, Habegger~\cite[Corollary~1.2]{Habegger}, and Bays--Hart--Pillay~\cite[Lemma~A.7]{BHP} revealed a more detailed structure of this $A(K)$.
    The Kummer-faithfulness for $\kbar(\sigma)$ with $\sigma \in G_k$ is not known.
    Note that we make no claim on fields that fall into the blanks in Table~\ref{table:PossibleStructure}, including their existence.
    Furthermore, we do not intend to claim that the non-freeness of $A(K) / A(K)_\tor$ implies the infiniteness of the rank of it.
\end{remark}

This paper is organized as follows: 
in Section~\ref{sect:Proof of MT1}, 
we prove Theorem~\ref{MT1}. 
In Section~\ref{sec-preliminaries}, we recall some previous results needed in order to prove the results on the second type of extensions.
Section~\ref{sec-FreenessOfMWGmodTor} is devoted to proving that the Mordell--Weil groups modulo torsion of semiabelian varieties over a finite extension of $\overline{K}[\sigma]$ are free.
We also deduce that the Mordell--Weil group of any semiabelian variety over such a field is the direct sum of a finite torsion subgroup and a free $\mathbb{Z}$-module of denumerable rank.
Section~\ref{sec-KFforKbarsigma} concerns the Kummer-faithfulness of $\overline{K}(\sigma)$ and of $\overline{K}[\sigma]$.
Since the Kummer-faithfulness means that the Mordell--Weil group does not contain groups like $\mathbb{Q}$ or $\mathbb{Q} / \mathbb{Z}$ as subgroups, our results are thought of as describing more detailed structures of the Mordell--Weil groups of semiabelian varieties over such fields.

Section~\ref{sec-FreenessOfMWGmodTor} and a part of Section~\ref{sec-KFforKbarsigma} are originally part of the PhD thesis~\cite{AsayamaThesis} of the first author.
Theorem 4.2.1 in the thesis, which claimed the (torally) Kummer-faithfulness property for $\overline{K}(\sigma)$, is excluded in this paper because its proof was incorrect.
Note that Theorem~\ref{mythm-SummaryOfKFnessForKbarsigma}~(2) ($=$~Theorem~\ref{mythm-TKF}) in this paper is the corrected (weaker) version of this assertion.


\section*{Acknowledgments}
The authors thank \mbox{Naganori} \mbox{Yamaguchi} for helpful discussions during the preparation of this paper.
They are grateful to 
\mbox{Benjamin} \mbox{Collas}, 
\mbox{Shinichi} \mbox{Mochizuki}, 
\mbox{Yoshiyasu} \mbox{Ozeki}, 
\mbox{Akio} \mbox{Tamagawa}, and 
\mbox{Shota} \mbox{Tsujimura}
for their useful comments on an early version of the manuscript.
This work was supported 
by the Research Institute for Mathematical Sciences,
an International Joint Usage/Research Center located in Kyoto University.
The second author was supported
by JSPS \mbox{KAKENHI} Grant Number~\mbox{JP23K03068}.

This paper has been published in Res.\ Number Theory \textbf{11} (2025), Paper No.~89, 18 pp.
(\href{https://doi.org/10.1007/s40993-025-00658-2}{\texttt{https://doi.org/10.1007/s40993-025-00658-2}})
Open Access funding provided by Institute of Science Tokyo.

\section{Proof of Theorem~\ref{MT1}}\label{sect:Proof of MT1}

We begin by proving two lemmas needed in the proof 
of Theorem~\ref{MT1}  ($=$~Theorem~\ref{MT1-sec2}). 
Let  $K$  be a complete discrete valuation field with perfect residue field. 
If $L/K$ is a finite Galois extension, 
then the Galois group  $G$  has two filtrations; 
the lower-numbering filtration $(G_u)_{u\in\R}$ and 
the upper-numbering filtration $(G^v)_{v\in\R}$. 
These are related by the Herbrand function $\phiLK$  defined by
\[
\phiLK(u)\ =\ \int_0^u\frac{dt}{(G_0:G_t)}
\]
and its inverse 
$\psiLK$ in such a way that
\[
G^v\ =\ G_{\psiLK(v)}\qquad\text{and}\qquad
G^{\phiLK(u)}\ =\ G_u
\]
(cf.~\cite[Chapter~IV, Section~3]{Serre}). 
If  $H$  is a subgroup of  $G$, we have
\[
H_u\ =\ H\cap G_u,
\]
and, if further  $H$  is normal, then
\[
(G/H)^v\ =\ G^vH/H.
\]
Using this property, one can define the upper-numbering filtration  
$(G^v)_{v\in\R}$  
for an arbitrary (not necessarily finite) Galois extension $L/K$ 
by defining  $G^v$  to be the projective limit of  
$\Gal(L'/K)^v$, where 
$L'/K$ runs over all finite Galois subextensions of  $L/K$.

For a Galois extension $L/K$ with Galois group $G$ 
and a real number $c$, we say 
that 
$L/K$  has {\it ramification bounded by} $c$ if 
$G^c=1$. 
Note that $L/K$  has ramification bounded by $c$ 
if and only if all finite Galois subextensions  $L'/K$  have 
ramification bounded by $c$.

\begin{lemma}\label{lem:ram.bound}
Let  
$L/K$  be a Galois extension and
$K'/K$  a finite Galois subextension of $L/K$. 
If both $L/K'$ and  $K'/K$  have ramification bounded by $c$, 
then $L/K$ has also ramification bounded by $c$.
\end{lemma}

\begin{proof}
Put 
$G=\Gal(L/K)$  and 
$H=\Gal(L/K')$. 
For each real number  $c$, we have an exact sequence
\[
1\ \to\ G^c\cap H\ \to\ G^c\ \to\ (G/H)^c\ \to\ 1.
\]
We have 
\begin{align*}
G^c\cap H\ 
&=\ 
G_{\psiLK(c)}\cap H\ =\ 
H_{\psiLK(c)}\\
&=\ 
H^{\phiLKp(\psiLK(c))} \ =\ 
H^{\psiKpK(c)}\ \subset H^c.
\end{align*}
Thus if both  $H^c=1$  and  $(G/H)^c=1$  hold, then we have 
$G^c=1$.
\end{proof}

\begin{lemma}\label{lem:conductor}
Assume that  $K$  has finite residue field 
with absolute ramification index  $\eK$ and residue characteristic $p_K$. 
If  $L/K$  is a finite abelian extension of exponent  $p^m$
for some prime number  $p$  and integer  $m\geq 1$, 
then it has ramification bounded by 
\[
\begin{cases}
1                                         &\text{if}\ p_K\not=p, \\
\eK\left(m+\dfrac{1}{p-1}\right) &\text{if}\ p_K=p.
\end{cases}
\]
\end{lemma}

If $K$ has characteristic $>0$, then we understand that  
$\eK=\infty$, so that the lemma trivially holds true. 

\begin{proof}
If $p_K\not=p$, then 
$L/K$ is tamely ramified and hence it has ramification bounded by $1$. 
Suppose that $K$ has residue characteristic  $p$. 
We may assume that  $L/K$  is totally ramified and cyclic of degree $p^m$. 
Let  $\OK$  denote the valuation ring of  $K$.  
By local class field theory, we have the reciprocity map
\[
\OK^\times\ \to\ \Gal(L/K), 
\]
which maps  $\UKu$  onto  $\Gal(L/K)^u$, 
where   $\UKu\coloneqq1+\mK^u$  for an integer  $u\geq 1$  and  
$\mK$  is the maximal ideal of  $\OK$. 
We claim that, if
$u>\eK(m+1/(p-1))$, then one has
\[
\UKu\ \subset\ (\OK^\times)^{p^m},
\]
so that any cyclic quotient of 
$\Gal(L/K)$  of degree dividing  $p^m$  has maximal ramification break $\leq u$. 
This follows from the equality
\[
1+x\ =\ 
\left(\exp\left(\frac{1}{p^m}\log(1+x)\right)\right)^{p^m},
\]
which holds for  $x\in\OK$  with  $v_K(x)>\eK(m+1/(p-1))$
(see, for example,~\cite[Chapter~II, Proposition~5.5]{Neukirch}).
\end{proof}


Now we prove our first main result.

\begin{theorem}\label{MT1-sec2}
Let $k$ be a number field.
For any integer $g\geq 1$ and any family $\m=(m_p)_{p:\textup{prime}}$ of integers 
$m_p\geq 0$ indexed by prime numbers, 
the extension field
\[
\Kgm \coloneqq k(p^{-m_p}A(k) \mid A, p)
\] 
of  $k$ 
obtained by adjoining all coordinates of the points in 
\[
p^{-m_p}A(k) \coloneqq 
\{P\in A(\kbar)\mid p^{m_p}P\in A(k)\}
\] 
for 
all semiabelian variety  $A$  over  $k$  of dimension $\leq g$  and 
all prime numbers  $p$ 
is highly Kummer-faithful.
\end{theorem}

\begin{proof}
By~\cite[Corollary~2.15]{OT}, 
it is enough to show that the Galois extension 
$\Kgm/k$  has finite maximal ramification break everywhere. 
Since the upper bound of the 
upper ramification break is preserved by 
composition of fields (\cite[Chapter~IV, Proposition~14]{Serre}; 
see also~\cite[Lemma~3.1]{OT}), 
it is enough to show that, for each finite place  $v$  of  $k$,  
there is a constant  $c_v=c_v(g,\m)$  depending only on  
$k$, $v$, $g$, and $\m$ 
such that, for any semiabelian variety  $A$  over  $k$  
of dimension $\leq g$  and any prime number  $p$, 
the Galois extension  
$k(p^{-m_p}A(k))/k$  has upper ramification break  $\leq c_v$ 
at any extension  $w$  of  $v$  to $k(p^{-m_p}A(k))$. 
Set
$L=k(p^{-m_p}A(k))$ and 
$K=k(A[p^{m_p}])$. 
By Lemma~\ref{lem:ram.bound}, 
it is enough to show that the maximal upper ramification breaks of  
$L/K$  and  $K/k$  at  $w$  are both  $\leq c$. 
Such a bound, say $\cvOT$, is given for  $K/k$  by~\cite[Theorem~3.3]{OT}. 
For the abelian extension 
$L/K$  of exponent  $\leq p^{m_p}$, 
Lemma~\ref{lem:conductor} assures that 
it has ramification bounded by 1 if  $w\nmid p$  and 
by  $e_{K,w}(m_p+1/(p-1))$  if  $w\mid p$, 
where  $e_{K,w}$  denotes the absolute ramification index of  $K$  at  $w$.
Note that, if $w\mid p$,  the absolute ramification index 
$e_{K,w}$  is bounded by a constant  $e_v$  depending only on 
$k$, $v$, $g$, and $m_p$. 
Now we obtain a desired upper bound by setting
\[
c_v \coloneqq 
\max\left\{\cvOT, e_v\left(m_p+\frac{1}{p-1}\right)\right\}.
\]
The proof of the theorem is completed.
\end{proof}

\begin{proposition}\label{prop:not free-sec2}
Let  $A$  be a semiabelian variety over  $k$  of dimension $\leq g$. 
If  $A(k)$  has a non-torsion point (i.e., if  $\rank(A(k))\geq 1$)  and  
$m_p\geq 1$  for infinitely many primes  $p$, then
the Mordell--Weil group  $A(\Kgm)$  modulo the 
torsion subgroup  $A(\Kgm)_\tor$  is {\rm not} free. 
In particular, it is not finitely generated. 
\end{proposition}

This follows from the following lemma, since 
if  $P\in A(k)$  is a non-torsion point, then it is divisible by  
$p^{m_p}$ ($>1$) in  $A(\Kgm)$  for infinitely many  $p$.

\begin{lemma}\label{lem:not free}
Let  $M$  be a $\Z$-module.
If there exist a non-zero element  $x$  of  $M$  and 
infinitely many positive integers  $d_1,d_2,\dots$  such that  
$x$  is divisible by  $d_j$  in  $M$  
(i.e., there exists a  $y_j\in M$  such that  $d_jy_j=x$) for all  $j\geq 1$,  
then  $M$  is not free.
\end{lemma}

\begin{proof}
Suppose that  $M$  is free and choose a basis 
$(m_i)_{i\in I}$  of it.
Let  $x\in M\smallsetminus\{0\}$  and  $d_1,d_2,\dots$  be as above. 
Write  
$x=\sum_{i\in I}a_im_i$  with  $a_i\in\Z$. 
Then  $x$  is divisible by  $d_j$  in  $M$  if and only if  
$a_i$  is divisible by  $d_j$  for all  $i\in I$. 
If this holds for infinitely many  $d_j$, 
then all  $a_i$  are zero. 
This is a contradiction, and hence  $M$  is not free.
\end{proof}

\begin{remark}
\begin{enumerate}[\textup{(\arabic{enumi})}]
    \item A $\Z$-module $M$  as in Lemma~\ref{lem:not free} 
    (and hence  $A(\Kgm)$  modulo torsion in Proposition~\ref{prop:not free-sec2})
    is even {\it not projective} by~\cite[Corollary~4.5]{Bass}.
    \item 
    In general, we say that the Mordell--Weil group  $A(K)$  
    over a field  $K$  has \textit{infinite rank} if 
    $\dim_\Q(A(K)\otimes_\Z\Q)=\infty$. 
    If $\rank(A(k))=0$, in many cases, 
    $A(\Kgm)$  and $A(\kgm)$  
    are known to have infinite rank. 
    For example, this is the case for both 
    $A(\Kgm)$ and $A(\kgm)$ if 
    $A$ is the Jacobian variety of a hyperelliptic curve of genus $\leq g$  defined over  $k$  and 
    $m_p\geq 1$  for infinitely many  $p$
    (cf.~\cite[Theorem~1]{Moon09}).
\end{enumerate}
\end{remark}

\section{Known results for our second type of extensions}
\label{sec-preliminaries}

Let $K$ be a field.
Let $\zeta_n$ denote a primitive $n$-th root of unity in $\overline{K}$ for any positive integer $n$.
For any $\mathbb{Z}$-module $M$, denote by $M_\divisible$ the submodule of divisible elements in $M$, i.e., $M_\divisible = \bigcap_{n = 1}^{\infty} n M$.
For later use in proving our results, let us summarize what has been found on the structure of the Mordell--Weil groups of semiabelian varieties over $\overline{K}(\sigma)$ and over $\overline{K}[\sigma]$.

\begin{theorem}\label{knownresultforKbarsigma}
Let $K$ be a finitely generated field over its prime field and $e$ a positive integer.
\begin{enumerate}[\textup{(\arabic{enumi})}]
    \item\textup{(Geyer--Jarden~\cite[Theorem~2.4]{GJ06} (resp.\ Frey--Jarden~\cite[Theorem~9.1]{FreyJ}))} Assume that $K$ is an infinite field.
    Then for almost all $\sigma \in G_K^e$ and any abelian variety $A$ of positive dimension over $\overline{K}[\sigma]$ (resp.\ $\overline{K}(\sigma)$), the group $A(\overline{K}[\sigma])$ (resp.\ $A(\overline{K}(\sigma))$) has rank~$\aleph_0$.
    \item\textup{(Jarden~\cite[Theorems~8.1 and 8.2]{Jarden75})}
    \begin{enumerate}[\textup{(\arabic{enumi}-\roman{enumii})}]
        \item Assume $e = 1$.
        For almost all $\sigma \in G_K$ and any positive integer $d$, there exist infinitely many prime numbers $\ell$ such that $[\overline{K}(\sigma)(\zeta_\ell) : \overline{K}(\sigma)] = d$.
        In particular, for almost all $\sigma \in G_K$ and any finite extension $M$ of $\overline{K}(\sigma)$, there exist infinitely many roots of unity contained in $M$ and a prime number $\ell$ such that $\zeta_\ell \notin M$.
        \item Assume $e \ge 2$.
        Then for almost all $\sigma \in G_K^e$ and any positive integer $d$, there exist only finitely many positive integers $n$ not divisible by the characteristic of $K$ such that $[\overline{K}(\sigma)(\zeta_n) : \overline{K}(\sigma)] \le d$.
        In particular, for almost all $\sigma \in G_K^e$ and any finite extension $M$ of $\overline{K}(\sigma)$, there are only finitely many roots of unity contained in $M$.
    \end{enumerate}
    \item Consider the following statements on $K$:
    \begin{enumerate}[\textup{(\alph{enumii})}]
        \item Assume $e = 1$.
        For almost all $\sigma \in G_K$ and any abelian variety $A$ of positive dimension over $\overline{K}(\sigma)$, the group $A(\overline{K}(\sigma))_\tor$ is infinite.
        Moreover, there exist infinitely many prime numbers $\ell$ such that $A(\overline{K}(\sigma))[\ell] \neq 0$.
        \item Assume $e \ge 2$.
        For almost all $\sigma \in G_K^e$ and any abelian variety $A$ over $\overline{K}(\sigma)$, the group $A(\overline{K}(\sigma))_\tor$ is finite.
        \item For almost all $\sigma \in G_K^e$, any abelian variety $A$ over $\overline{K}(\sigma)$, and any prime number $\ell$, the group $A(\overline{K}(\sigma))[\ell^\infty] = \bigcup_{i = 1}^{\infty} A(\overline{K}(\sigma))[\ell^i]$ is finite.
    \end{enumerate}
    Then these statements hold in the following situations:
    \begin{enumerate}[\textup{(\arabic{enumi}-\roman{enumii})}]
    \item\textup{(Geyer--Jarden~\cite[Theorem~1.1]{GJ78})} Replace ``any abelian variety" in each statement with ``any elliptic curve".
    Then Statements~\textup{(a)--(c)} hold for any $K$.
    \item\textup{(Jacobson--Jarden~\cite[Proposition~4.2]{JJ84})} Statements~\textup{(a)--(c)} hold if $K$ is a finite field.
    \item\textup{(Jacobson--Jarden~\cite[Main Theorem~(b), (a)]{JJ01} (proved (b) and (c)), Zywina~\cite[Theorem~1.1]{Zywina16} (proved (a) for the number field case), Jarden--Petersen~\cite[Theorem~C]{JP19} (proved (a) for the general case))} Statements~\textup{(a)--(c)} hold if $K$ has characteristic zero.
    \item\textup{(Jacobson--Jarden~\cite[Main Theorem~(a)]{JJ01})} Statement~\textup{(c)} holds for any $K$.
    \end{enumerate}
    \item\textup{(Jarden--Petersen~\cite[Theorem~1.3 (ii)]{JP22})} Assume that $K$ has characteristic zero and $e \ge 2$.
    Then for almost all $\sigma \in G_K^e$, any finite extension $M$ of $\overline{K}(\sigma)$, and any abelian variety $A$ over $M$, it holds that $A(M)_\divisible = 0$.
\end{enumerate}
\end{theorem}

The results, including ours, involving the normalized Haar measure are non-constructive as they are proved by calculating the measures of various subsets in $G_K^e$.
Thus they do not give an explicit element for which their statement holds.
Larsen~\cite{Larsen03} conjectured that the statement on the field $\overline{K}(\sigma)$ in (1) holds for \textit{any} $\sigma \in G_K^e$.
He and Im~\cite[Theorem~1.4]{IL08} gave an affirmative answer to this conjecture under the assumptions that $K$ is of characteristic different from two and $e = 1$.
Im with collaborators~\cite{BI, IL13sa, CI} has obtained other results concerning Larsen's conjecture.
See also the survey~\cite{IL21} by Im and Larsen.

Geyer and Jarden~\cite{GJ78} conjectured that Statements (a)--(c) in (3) in this theorem hold for any finitely generated field $K$ over its prime field.
We note that the paper of Jacobson and Jarden~\cite{JJ84} involves a proof of Statement~(a) for $K$ with positive characteristic, but it contains an error as indicated in~\cite{JJ85}.
Statements~(a) and (b) for $K$ which is infinite and has positive characteristic remain open.

We can extend (3-iii) in Theorem~\ref{knownresultforKbarsigma} to finite extensions of $\overline{K}(\sigma)$.
For the convenience of applying this theorem in the later sections, we describe the assertion of this theorem again.

\begin{theorem}\label{torptsoffinextofKbarsigma}
Let $K$ be a finitely generated field over $\mathbb{Q}$ and $e$ a positive integer.
\begin{enumerate}[\textup{(\arabic{enumi})}]
    \item Assume $e \ge 2$.
    For almost all $\sigma \in G_K^e$, any finite extension $M$ of $\overline{K}(\sigma)$, and any abelian variety $A$ over $M$, the group $A(M)_\tor$ is finite.
    \item For almost all $\sigma \in G_K^e$, any finite extension $M$ of $\overline{K}(\sigma)$, any abelian variety $A$ over $M$, and any prime number $\ell$, the group $A(M)[\ell^\infty]$ is finite.
\end{enumerate}
\end{theorem}

\begin{proof}
Let $A$ be an abelian variety over a finite extension $M$ of $\overline{K}(\sigma)$.
Using Weil restriction for abelian varieties~\cite[Lemma~6.1]{JP22}, we know that $B = \mathrm{Res}_{M / \overline{K}(\sigma)}(A)$ is an abelian variety over $\overline{K}(\sigma)$ and $A(M) \cong B(\overline{K}(\sigma))$.
Hence the theorem follows from (3-iii) in Theorem~\ref{knownresultforKbarsigma}.
\end{proof}

The next proposition provides a convenient criterion for a perfect field to be Kummer-faithful.

\begin{proposition}\label{CharacterizationOfKFfields}
A perfect field $K$ is Kummer-faithful if and only if $\Gm(L)_\divisible = 0$ for any finite extension $L$ of $K$ (i.e., $K$ is torally Kummer-faithful) and $A(K)_\divisible = 0$ for any abelian variety $A$ over $K$.
\end{proposition}

\begin{proof}
See~\cite[Proposition~2.3]{OT}.
\end{proof}

\section{Freeness of Mordell--Weil groups modulo torsion}
\label{sec-FreenessOfMWGmodTor}

The aim of this section is to prove the freeness result for the Mordell--Weil groups over finite extensions of $\overline{K}[\sigma]$ modulo torsion.
We note that, because these fields are countable, the rank of the Mordell--Weil groups over such a field has infinite rank if and only if it has rank $\aleph_0$.

\begin{theorem}\label{mythm-FreenessOfMWGmodTor}
Suppose that $K$ is a finitely generated field over $\mathbb{Q}$ and $e \ge 2$.
Then, for almost all $\sigma \in G_K^e$, the following statement holds: for any finite extension $L$ of $\overline{K}[\sigma]$ and any semiabelian variety $A$ of positive dimension over $L$, the group $A(L) / A(L)_\tor$ is a free $\mathbb{Z}$-module of rank $\aleph_0$.
\end{theorem}

Before proving this theorem, let us recall the proposition by Moon~\cite{Moon09}, which plays a key role in our proof.
It is notable that Moon seemingly proved this proposition only for the case where $K$ is a number field and $A$ is an abelian variety, but the same proof works in the more general setting.

\begin{proposition}[Moon~{\cite[Proposition~7]{Moon09}}]\label{MoonProp}
Let $K$ be a field of cardinality at most $\aleph_0$ and $A$ a semiabelian variety over $K$.
Let $L$ be a Galois extension of $K$ such that $A(L)_\tor$ is finite.
Then the group $A(L) / A(L)_\tor$ is a free $\mathbb{Z}$-module of rank at most $\aleph_0$.
\end{proposition}

\begin{proof}[Proof of Theorem~\textup{\ref{mythm-FreenessOfMWGmodTor}}]
We separate the proof of this theorem into two parts, one for the freeness of the group $A(L) / A(L)_\tor$ and the other for the infiniteness of the rank of the group $A(L)$.
First, we show the former part.

Let $\sigma \in G_K^e$ satisfy the following: for any finite extension $M$ of $\overline{K}(\sigma)$, only finitely many roots of unity belong to $M$ and the group $B(M)_\tor$ is finite for any abelian variety $B$ over $M$.
Since $e \ge 2$, almost all $\sigma \in G_K^e$ satisfy this condition by Theorems~\ref{knownresultforKbarsigma}~(2-ii) and \ref{torptsoffinextofKbarsigma}~(1).
Let $L$ be a finite extension of $\overline{K}[\sigma]$ and $M = L \cdot \overline{K}(\sigma)$.
Then $M$ is a finite extension of $\overline{K}(\sigma)$.
Let $A$ be a semiabelian variety over $L$.
Then $A$ is an extension of an abelian variety $B$ by a torus $T$.
By assumption, the groups $T(M)_\tor$ and $B(M)_\tor$ are finite, and so are $A(M)_\tor$ and $A(L)_\tor$.
There exists a finite extension $K'$ of $K$ in $L$ such that $L / K'$ is Galois and $A$ is defined over $K'$.
Applying Proposition~\ref{MoonProp} to $L / K'$ and $A$, we find that $A(L) / A(L)_\tor$ is a free $\mathbb{Z}$-module of rank at most $\aleph_0$.

It remains to show that almost all $\sigma \in G_K^e$ satisfy the following condition: for any finite extension $L$ of $\overline{K}[\sigma]$ and any semiabelian variety $A$ of positive dimension over $L$, the group $A(L)$ has infinite rank.
In fact, it turns out that we only need to prove this when $A$ is an abelian variety of positive dimension.
Theorem~\ref{knownresultforKbarsigma}~(1) says that the following weaker claim than this statement holds for almost all $\sigma \in G_K^e$: for any abelian variety $A$ of positive dimension over $\overline{K}[\sigma]$, the group $A(\overline{K}[\sigma])$ has rank $\aleph_0$.
Let $\sigma$ satisfy the statement in the above claim, $L$ be a finite extension of $\overline{K}[\sigma]$, and $A$ an abelian variety of positive dimension over $L$.
Let $B = \mathrm{Res}_{L / \overline{K}[\sigma]}(A)$ be the Weil restriction of $A$ with respect to $L / \overline{K}[\sigma]$.
Then $B$ is an abelian variety over $\overline{K}[\sigma]$ and we have $A(L) \cong B(\overline{K}[\sigma])$ by~\cite[Lemma~6.1]{JP22}.
The assumption on $\sigma$ implies that $B(\overline{K}[\sigma])$ has rank $\aleph_0$ and $A(L)$ also does, which completes the proof.
\end{proof}

As described in the proof, for almost all $\sigma \in G_K^e$, any finite extension $L$ of $\overline{K}[\sigma]$, and any semiabelian variety $A$ over $L$, the torsion group $A(L)_\tor$ is finite.
Combining this remark with the theorem, we obtain the structure of the group $A(L)$.

\begin{corollary}\label{mycor-FreenessOfMWGmodTor}
Let $K$ be a finitely generated field over $\mathbb{Q}$ and $e \ge 2$.
Then, for almost all $\sigma \in G_K^e$, the following statement holds: for any finite extension $L$ of $\overline{K}[\sigma]$ and any semiabelian variety $A$ over $L$, the group $A(L)$ is the direct sum of a finite torsion subgroup and a free $\mathbb{Z}$-module of rank $\aleph_0$.
\end{corollary}

\begin{proof}
Let $\sigma \in G_K^e$ satisfy each statement in Theorems~\ref{knownresultforKbarsigma}~(2-ii), \ref{torptsoffinextofKbarsigma}~(1), and~\ref{mythm-FreenessOfMWGmodTor}.
We show that the statement in the corollary holds for $\sigma$.
Let $L$ be a finite extension of $\overline{K}[\sigma]$ and $A$ a semiabelian variety over $L$.
Then the group $A(L) / A(L)_\tor$ is a free $\mathbb{Z}$-module of rank $\aleph_0$ and in particular it is projective.
Hence the identity map on $A(L) / A(L)_\tor$ can be lifted to a homomorphism $A(L) / A(L)_\tor \to A(L)$.
This provides a section of the exact sequence
\[ 0 \to A(L)_\tor \to A(L) \to A(L) / A(L)_\tor \to 0 \]
and we have $A(L) = A(L)_\tor \oplus A(L) / A(L)_\tor$.
The corollary follows from this decomposition.
\end{proof}

\begin{remark}
If $e = 1$, then the proof of Theorem~\ref{mythm-FreenessOfMWGmodTor} is invalid.
This is because, for almost all $\sigma \in G_K$ and any abelian variety $A$, the groups $\Gm(\overline{K}(\sigma))_\tor$ and $A(\overline{K}(\sigma))_\tor$ are infinite (Theorem~\ref{knownresultforKbarsigma}, (2-i) and (3-iii)).
It is not known whether Theorem~\ref{mythm-FreenessOfMWGmodTor} still holds in the case $e = 1$.
We note that $\Gm(\overline{K}(\sigma)) / \Gm(\overline{K}(\sigma))_\tor$ is not free for almost all $\sigma \in G_K$ (see Remark~\ref{rmk-TKFness}).

We also mention the following facts.
Let $K$ be a finitely generated field over $\mathbb{Q}$ and $e$ a positive integer.
Then, for almost all $\sigma \in G_K^e$, the field $\overline{K}(\sigma)$ is a Galois extension of no proper subfield of $\overline{K}(\sigma)$ and $\overline{K}(\sigma) / \overline{K}[\sigma]$ is an infinite extension.
These facts follow from~\cite[Theorems~7.9 and 7.10]{BS}.
\end{remark}

\section{Kummer-faithfulness for some large algebraic extensions}
\label{sec-KFforKbarsigma}

In the present section, we investigate when the fields $\overline{K}(\sigma)$ and $\overline{K}[\sigma]$ are Kummer-faithful for a finitely generated field $K$ over $\mathbb{Q}$ and $\sigma \in G_K^e$.
The following proposition suggests that these fields should be of interest since every sub-$p$-adic field is Kummer-faithful.

\begin{proposition}\label{prop-NotSubpAdic}
Let $K$ be a finitely generated field over $\mathbb{Q}$ and $e$ a positive integer.
Then neither $\overline{K}(\sigma)$ nor $\overline{K}[\sigma]$ are sub-$p$-adic for almost all $\sigma \in G_K^e$ and any prime number $p$.
\end{proposition}

\begin{proof}
Since $\overline{K}(\sigma)$ contains $\overline{K}[\sigma]$, we only have to discuss whether $\overline{K}[\sigma]$ is sub-$p$-adic or not.
We first consider the case where $K = k$ is a number field.
It is known that $\overline{k}[\sigma]$ is \textit{pseudo algebraically closed} (see~\cite[Chapter~12]{FriedJ} for the definition) for almost all $\sigma \in G_k^e$~\cite[Theorem~21.12.2]{FriedJ}.
Moreover, an algebraic extension of a pseudo algebraically closed field is henselian with respect to some nonarchimedean place if and only if it is the algebraic closure~\cite[Corollary~12.5.6]{FriedJ}.
Then we have $\overline{k} \cap (\overline{k}[\sigma])_\mathfrak{p} = \overline{k}$ for almost all $\sigma \in G_k^e$ and any nonarchimedean place $\mathfrak{p}$ of $\overline{k}[\sigma]$.
If $\overline{k}[\sigma]$ is sub-$p$-adic, then it is a subfield of a finite extension of $\mathbb{Q}_p$.
However, since $\overline{\mathbb{Q}} \cap L \neq \overline{\mathbb{Q}}$ for any finite extension $L$ of $\mathbb{Q}_p$, the field $\overline{k}[\sigma]$ is not sub-$p$-adic for almost all $\sigma \in G_k^e$.

For a general $K$, let $k$ be the algebraic closure of $\mathbb{Q}$ in $K$.
Then $k$ is a number field.
The restriction map $G_K \to G_k$ is surjective and $\overline{k}[\sigma|_{\overline{k}}] \subset \overline{K}[\sigma]$ for any $\sigma \in G_K^e$.
As we have shown above, the field $\overline{k}[\sigma|_{\overline{k}}]$ is not sub-$p$-adic for almost all $\sigma \in G_K^e$.
For such $\sigma$, the field $\overline{K}[\sigma]$ is also not sub-$p$-adic.
\end{proof}


\begin{theorem}\label{mythm-KFforKbarbrasigma}
Let $K$ be a finitely generated field over $\mathbb{Q}$ and $e$ a positive integer.
Then, for almost all $\sigma \in G_K^e$, any finite extension of $\overline{K}[\sigma]$ is Kummer-faithful.
\end{theorem}

\begin{proof}
If $e \ge 2$, then the proposition follows from Corollary~\ref{mycor-FreenessOfMWGmodTor}.
Hence we may assume $e = 1$.
Since Kummer-faithfulness is preserved under finite extensions, we only discuss whether $\overline{K}[\sigma]$ is Kummer-faithful.
By Proposition~\ref{CharacterizationOfKFfields}, it suffices to show that the following two statements hold for almost all $\sigma \in G_K$:
\begin{enumerate}[(\alph{enumi})]
    \item $\Gm(L)_\divisible = 0$ for any finite extension $L$ of $\overline{K}[\sigma]$;
    \item $A(\overline{K}[\sigma])_\divisible = 0$ for any abelian variety $A$ over $\overline{K}[\sigma]$.
\end{enumerate}

Since $L$ is a Galois extension of some finite extension of $K$, the conditions $\Gm(L)_\divisible = 0$ and $A(\overline{K}[\sigma])_\divisible = 0$ respectively can be replaced with $(\Gm(L)_\tor)_\divisible = 0$ and that $A(\overline{K}[\sigma])[\ell^\infty]$ is finite for any prime number $\ell$~\cite[Proposition~2.4~(2)]{OT}.
Then Statement~(a) holds for almost all $\sigma \in G_K$ from Theorem~\ref{knownresultforKbarsigma}~(2-i).
Statement~(b) holds for almost all $\sigma \in G_K$ by Theorem~\ref{torptsoffinextofKbarsigma}~(2).
\end{proof}

In the remainder of this paper, we discuss the Kummer-faithfulness for the field $\overline{K}(\sigma)$.
By Proposition~\ref{CharacterizationOfKFfields}, we can divide the discussion into two parts; torally Kummer-faithfulness and the vanishing property of the divisible part of the Mordell--Weil groups of abelian varieties.
Jarden and Petersen~\cite[Remark~5.5]{JP22}, who proved the latter part when $e \ge 2$ (cf.\ Theorem~\ref{knownresultforKbarsigma}~(4)), pointed out that their proof does not work when $e = 1$.
We do not know at the time of writing this paper whether the Kummer-faithfulness property holds for any $e \ge 1$.
Here we give a partial result on the torally Kummer-faithfulness for any $e \ge 1$.

\begin{theorem}\label{mythm-TKF}
Suppose that $K$ is a finitely generated field over $\mathbb{Q}$ and $e$ is a positive integer.
Then, the group $\Gm(\overline{K}(\sigma))_\divisible$ is trivial for almost all $\sigma \in G_K^e$.
\end{theorem}

\begin{remark}\label{OpenProblemForKFness}
If we accepted the claim in~\cite[Theorem~1.11]{Mochizuki15} that the absolute Galois group of any Kummer-faithful field of characteristic zero is \textit{slim} (i.e., every open subgroup has trivial center), then the field $\overline{K}(\sigma)$ would \textit{not} be Kummer-faithful for \textit{any} $\sigma \in G_K$ since the absolute Galois group $G_{\overline{K}(\sigma)}$ of $\overline{K}(\sigma)$ is abelian; $G_{\overline{K}(\sigma)}$ is the closed subgroup $\langle\sigma\rangle$ in $G_K$ generated by $\sigma$.
However, Mochizuki recently informed us that there is a gap in the proof of this claim and the status remains open for general Kummer-faithful fields of characteristic zero.
If the Kummer-faithfulness result for $\overline{K}(\sigma)$ holds when $e = 1$, then it leads us to conclude that the above claim does not hold in the general case, although it provides no explicit counterexample.
It would also answer in the negative the questions~\cite[Remark~2.4.1, Questions~1 and~2]{MT} posed by Minamide and Tsujimura, which ask whether the absolute Galois group of any torally Kummer-faithful field is slim.
\end{remark}

\begin{remark}
For a finite field $\mathbb{F}$, the Kummer-faithfulness result on $\overline{\mathbb{F}}(\sigma)$ for $\sigma \in G_{\mathbb{F}}^e$ holds for every positive integer $e$.
Indeed, if $\sigma = (\sigma_1, \ldots, \sigma_e) \in G_{\mathbb{F}}^e$, then $\overline{\mathbb{F}}(\sigma) = \bigcap_{i = 1}^e \overline{\mathbb{F}}(\sigma_i)$ and it suffices to show the claim for $e = 1$.
In this case, the absolute Galois group $G_{\overline{\mathbb{F}}(\sigma)}$ of $\overline{\mathbb{F}}(\sigma)$ is the closed subgroup $\langle\sigma\rangle$ generated by $\sigma$ in $G_{\mathbb{F}}$ and it is isomorphic to $\widehat{\mathbb{Z}}$ for almost all $\sigma \in G_{\mathbb{F}}$~\cite[Lemma~7.1~(b)]{Jarden75}.
For such $\sigma$, the isomorphism $G_{\overline{\mathbb{F}}(\sigma)} \cong \widehat{\mathbb{Z}}$ implies the Kummer-faithfulness of $\overline{\mathbb{F}}(\sigma)$~\cite[Theorem~B~(ii)]{Murotani}.
\end{remark}

It is obvious that we may assume $e = 1$ to prove Theorem~\ref{mythm-TKF}.
Before proving the theorem, we need a lemma concerning the degree of the splitting field of the polynomial $X^n - a$ with some conditions on $a \in K$ and a positive integer $n$.

A family $\{B_i\}_{i \in I}$ of measurable subsets of $G_K$ is \textit{$\mu$-independent} if $\mu(\bigcap_{i \in J} B_i) = \prod_{i \in J} \mu(B_i)$ for every finite subset $J$ in $I$.
A family $\{L_i\}_{i \in I}$ of finite separable extensions of a field $K$ is \textit{linearly disjoint} over $K$ if $[K(\bigcup_{i \in J} L_i) : K] = \prod_{i \in J} [L_i : K]$ for every finite subset $J$ in $I$.
This occurs if and only if the family $\{G_{L_i}\}_{i \in I}$ of open subgroups of $G_K$ is $\mu$-independent~\cite[Lemma~21.5.1]{FriedJ}.

\begin{lemma}\label{lem-LDnessandExt}
Let $k$ be a number field.
There exists a positive integer $m_0$ such that the family $\{k\} \cup \{\mathbb{Q}(\zeta_\ell)\}_{\ell \ge m_0:\mathrm{prime}}$ is linearly disjoint over $\mathbb{Q}$.
\end{lemma}

\begin{proof}
We know that $\mathbb{Q}(\zeta_\ell)$ for all prime numbers $\ell$ are linearly disjoint over $\mathbb{Q}$.
Hence the lemma follows from~\cite[Lemma~3.1.10]{FriedJ}.
\end{proof}

Let $\varphi$ denote the Euler totient function, i.e., $\varphi(n)$ for a positive integer $n$ is the number of positive integers less than or equal to $n$ that are coprime to $n$.

\begin{lemma}\label{lem-mathcalLa}
Let $K$ be a finitely generated field over $\mathbb{Q}$ and $a$ an element in $K^\times$ which is not a root of unity.
Then there exists a set $\Lambda_a$ of prime numbers such that:
\begin{itemize}
    \item $\sum_{\ell \in \Lambda_a} 1 / \ell = \infty$;
    \item if $n = \ell_1 \cdots \ell_r$, where $\ell_1, \ldots, \ell_r$ are distinct prime numbers in $\Lambda_a$, then the splitting field of $X^n - a$ over $K$ is of degree $n \varphi(n)$ over $K$.
\end{itemize}
\end{lemma}

\begin{proof}
Since $a$ is not a root of unity, we can take a prime number $\ell_0$ such that $a$ is not an $\ell$-th power in $K$ for all prime numbers $\ell \ge \ell_0$.
Let $k$ be the algebraic closure of $\mathbb{Q}$ in $K$.
By Lemma~\ref{lem-LDnessandExt}, there exists a positive integer $m_0$ such that $k$ and $\mathbb{Q}(\zeta_\ell)$ for all prime numbers $\ell \ge m_0$ are linearly disjoint over $\mathbb{Q}$.
We show that the set of prime numbers $\ge \max\{\ell_0, m_0, 3\}$ has the desired property.
The first condition is automatically satisfied since all but finitely many prime numbers belong to this set.
Let $n = \ell_1 \cdots \ell_r$, where $\ell_1, \ldots, \ell_r$ are distinct prime numbers $\ge \max\{\ell_0, m_0, 3\}$.
Since $\mathbb{Q}(\zeta_{\ell_1}), \ldots, \mathbb{Q}(\zeta_{\ell_r})$, and $k$ are linearly disjoint over $\mathbb{Q}$, we have $[k(\zeta_n) : k] = [\mathbb{Q}(\zeta_n) : \mathbb{Q}] = \varphi(n)$.
We also find $[K(\zeta_n) : K] = [k(\zeta_n) : k] = \varphi(n)$ since $K \cap k(\zeta_n) = k$.
Note that $n$ is odd as $\ell_i \ge 3$ for all $i$.
The lemma now follows from~\cite[Chapter~VI, Theorem~9.4]{Lang02}.
\end{proof}

Let $S$ be the set of $\sigma \in G_K$ for which the statement in Theorem~\ref{mythm-TKF} does not hold.
Then what we need to show is that $S$ has measure zero in $G_K$.
For any $a \in \overline{K}^\times \smallsetminus \{1\}$, let
\[ S_a = \{\sigma \in G_K \mid a \in \Gm(\overline{K}(\sigma))_\divisible\}. \]
Then we have
\[ S = \bigcup_{a \in \overline{K}^\times \smallsetminus \{1\}} S_a. \]
Since the set $\overline{K}^\times \smallsetminus \{1\}$ is countable, it suffices to show that $S_a$ has measure zero for any $a \in \overline{K}^\times \smallsetminus \{1\}$.

Suppose that $a$ is a root of unity.
Remark that a root $a \neq 1$ of unity belonging to $\Gm(M)_\divisible$ for a field $M$ is equivalent to that $M$ contains all $\ell$-power roots of unity for some prime number $\ell$.
We prove a stronger result than we need here.

\begin{proposition}\label{prop-lpowerunity}
Suppose that $K$ is a finitely generated field over $\mathbb{Q}$.
Then, for almost all $\sigma \in G_K$, any finite extension $M$ of $\overline{K}(\sigma)$, and any prime number $\ell$, the multiplicative group $M^\times$ does not contain all $\ell$-power roots of unity.
\end{proposition}

\begin{proof}
By Steps~1 and~2 in~\cite[Section~8]{Jarden75}, we only consider the case where $K = \mathbb{Q}$.
In this case, it is sufficient to show that, for any prime number $\ell$ and any positive integer $N$, the set
\[ E(N, \ell) = \{\sigma \in G_\mathbb{Q} \mid [\overline{\mathbb{Q}}(\sigma)(\zeta_{\ell^\infty}) : \overline{\mathbb{Q}}(\sigma)] \le N\} \]
has measure zero in $G_\mathbb{Q}$.
Here $\overline{\mathbb{Q}}(\sigma)(\zeta_{\ell^\infty}) = \bigcup_{m = 1}^\infty \overline{\mathbb{Q}}(\sigma)(\zeta_{\ell^m})$.
We will prove this by showing that the measure of the set
\[ E(N, \ell, m) = \{\sigma \in G_\mathbb{Q} \mid [\overline{\mathbb{Q}}(\sigma)(\zeta_{\ell^m}) : \overline{\mathbb{Q}}(\sigma)] \le N\} \]
tends to zero as $m \to \infty$.
Suppose that $\ell$ is odd.
By~\cite[Lemma~3.1]{Jarden75}, we have
\[ \mu(E(N, \ell, m)) = \frac{1}{\varphi(\ell^m)} \sum_{d \le N; d \mid \varphi(\ell^m)} \varphi(d) \le \frac{1}{\ell^{m - 1} (\ell - 1)} \sum_{d \le N} \varphi(d). \]
Since the rightmost hand side goes to zero as $m \to \infty$, we have $\mu(E(N, \ell)) = 0$.
Suppose that $\ell = 2$.
We use \cite[Lemma~5.1]{Jarden75} with $e = 1, \theta = 1 / 3$.
Then we have
\[ \mu(E(N, 2, m)) \le \mu(E(2^{(m - 1) / 3}, 2, m)) \le \frac{c}{2^{m / 3}}, \]
where the first inequality holds for all $m \ge 3 \log N / \log 2 + 1$ and $c$ is a positive constant independent of $m$.
Taking $m \to \infty$ leads us to conclude that $\mu(E(N, 2)) = 0$.
\end{proof}

Suppose that $a$ is not a root of unity.
For a prime number $\ell$, put
\[ T_a^{(\ell)} = \{\sigma \in G_K \mid \text{some $\ell$-th root of $a$ belongs to }\overline{K}(\sigma)\}. \]
We see that
\[ S_a \subset \bigcap_{\ell\text{: prime}} T_a^{(\ell)}. \]
If we prove that, for almost all $\sigma \in G_K$, there exists a prime number $\ell$ such that $\sigma \notin T_a^{(\ell)}$, then it holds that $S_a$ has measure zero and the proof is completed.
It is easily seen that each $\sigma \in T_a^{(\ell)}$ fixes $a$, that is, $T_a^{(\ell)} \subset G_{K(a)}$.
Replacing $K$ with $K(a)$, we may assume that $a \in K^\times$.

Let $\alpha^{(\ell)}_1, \ldots, \alpha^{(\ell)}_\ell$ be all solutions of $X^\ell = a$ in $\overline{K}$ and $T_a^{(\ell, i)} (1 \le i \le \ell)$ be the set of $\sigma \in G_K$ such that $\alpha^{(\ell)}_i \in \overline{K}(\sigma)$.
Then we have $T_a^{(\ell)} = \bigcup_{i = 1}^\ell T_a^{(\ell, i)}$ and $\bigcap_{i = 1}^\ell T_a^{(\ell, i)} = G_{K(\sqrt[\ell]{a})}$, where $K(\sqrt[\ell]{a}) = K(\{\alpha_i^{(\ell)} \mid 1 \le i \le \ell\})$.

Let $\Lambda_a$ be the set of prime numbers obtained from Lemma~\ref{lem-mathcalLa}.

\begin{lemma}\label{lem-MeasureofTal}
We have $\mu(T_a^{(\ell)}) = 1 - 1 / \ell$ for all $\ell \in \Lambda_a$.
\end{lemma}

\begin{proof}
Note that $K$ contains none of $\alpha_i^{(\ell)}$ for all $i$.
Then
\[ \mu\left(T_a^{(\ell, i)}\right) = \frac{1}{\Big[K(\alpha_i^{(\ell)}) : K\Big]} = \frac{1}{\ell}. \]
Now we have
\[ T_a^{(\ell, j)} \cap T_a^{(\ell, j')} = G_{K(\sqrt[\ell]{a})} \]
for $j \neq j'$.
Indeed, any $\sigma$ in the left hand side fixes $\alpha_j^{(\ell)} / \alpha_{j'}^{(\ell)}$, which is a primitive $\ell$-th root of unity.
Thus $\sigma$ fixes $\alpha_i^{(\ell)}$ for all $i$, which implies $\sigma \in G_{K(\sqrt[\ell]{a})}$.
The inverse inclusion is obvious.

Therefore the $\ell + 1$ sets
\[ T_a^{(\ell, 1)} \smallsetminus G_{K(\sqrt[\ell]{a})}, \ldots, T_a^{(\ell, \ell)} \smallsetminus G_{K(\sqrt[\ell]{a})}, \text{ and } G_{K(\sqrt[\ell]{a})} \]
give a partition of $T_a^{(\ell)}$.
We calculate
\[ \mu\left(G_{K(\sqrt[\ell]{a})}\right) = \frac{1}{[K(\sqrt[\ell]{a}) : K]} = \frac{1}{\ell \varphi(\ell)} = \frac{1}{\ell (\ell - 1)}. \]
Hence
\begin{align*}
    \mu\left(T_a^{(\ell)}\right) &= \mu\left(\bigcup_{i = 1}^\ell \left(T_a^{(\ell, i)} \smallsetminus G_{K(\sqrt[\ell]{a})}\right) \cup G_{K(\sqrt[\ell]{a})}\right) \\
    &= \sum_{i = 1}^\ell \mu\left(T_a^{(\ell, i)} \smallsetminus G_{K(\sqrt[\ell]{a})}\right) + \mu\left(G_{K(\sqrt[\ell]{a})}\right) \\
    &= \sum_{i = 1}^\ell \mu\left(T_a^{(\ell, i)}\right) - (\ell - 1) \mu\left(G_{K(\sqrt[\ell]{a})}\right) \\
    &= \ell \cdot \frac{1}{\ell} - (\ell - 1) \cdot \frac{1}{\ell (\ell - 1)} = 1 - \frac{1}{\ell},
\end{align*}
as desired.
\end{proof}

\begin{lemma}\label{lem-independence}
The family $\{T_a^{(\ell)}\}_{\ell \in \Lambda_a}$ is $\mu$-independent.
\end{lemma}

\begin{proof}
Let $\ell_1, \ldots, \ell_r$ be distinct prime numbers in $\Lambda_a$ and $n = \ell_1 \cdots \ell_r$.
Then
\begin{align*}
    \mu\left(\bigcap_{i = 1}^r G_{K(\sqrt[\uproot{3}\ell_i]{a})}\right) &= \mu\left(\vphantom{G_{K(\sqrt[\uproot{3}l_i]{a})}}G_{K(\sqrt[n]{a})}\right) = \frac{1}{[K(\sqrt[n]{a}) : K]} = \frac{1}{n \varphi(n)} \\
    &= \prod_{i = 1}^r \frac{1}{\ell_i \varphi(\ell_i)} = \prod_{i = 1}^r \mu\left(G_{K(\sqrt[\uproot{3}\ell_i]{a})}\right).
\end{align*}
Hence $\{G_{K(\sqrt[\ell]{a})}\}_{\ell \in \Lambda_a}$ is $\mu$-independent.
We notice that whether $\sigma \in G_K$ belongs to $T_a^{(\ell)}$ determines only by $\sigma(\alpha_i^{(\ell)})$ for each $i$.
This implies that there exists a subset $B_\ell \subset \Gal(K(\sqrt[\ell]{a}) / K)$ such that $T_a^{(\ell)} = \{\sigma \in G_K \mid \sigma|_{K(\sqrt[\ell]{a})} \in B_\ell\}$.
By~\cite[Lemma~21.3.7]{FriedJ}, the family $\{T_a^{(\ell)}\}_{\ell \in \Lambda_a}$ is $\mu$-independent.
\end{proof}

\begin{proof}[Proof of Theorem~\textup{\ref{mythm-TKF}}]
By Lemma~\ref{lem-independence}, the family $\{T_a^{(\ell)}\}_{\ell \in \Lambda_a}$ is $\mu$-independent.
Using Lemma~\ref{lem-MeasureofTal}, we have
\[ \mu\left(\bigcap_{\ell\text{: prime}} T_a^{(\ell)}\right) \le \mu\left(\bigcap_{\ell \in \Lambda_a} T_a^{(\ell)}\right) = \prod_{\ell \in \Lambda_a} \mu\left(T_a^{(\ell)}\right) = \prod_{\ell \in \Lambda_a} \left(1 - \frac{1}{\ell}\right) = 0. \]
Therefore we conclude that there exists $\ell \in \Lambda_a$ such that $\sigma \notin T_a^{(\ell)}$ for almost all $\sigma \in G_K$.
\end{proof}

\begin{remark}\label{rmk-TKFness}
Let $K$, $a$, and $\Lambda_a$ be as in the discussion before Lemma~\ref{lem-MeasureofTal}.
Then we have $\sum_{\ell \in \Lambda_a} \mu(T_a^{(\ell)}) = \infty$.
The Borel--Cantelli lemma~\cite[Lemma~21.3.5]{FriedJ} in probability theory shows that, for almost all $\sigma \in G_K$, there exist infinitely many prime numbers $\ell \in \Lambda_a$ such that $\sigma \in T_a^{(\ell)}$.
Lemma~\ref{lem:not free} implies that the torsion-free $\mathbb{Z}$-module $\Gm(\overline{K}(\sigma)) / \Gm(\overline{K}(\sigma))_\tor$ is not free for almost all $\sigma \in G_K$.
\end{remark}


\bibliographystyle{amsalpha}

\begin{thebibliography}{GHP15}
\bibitem[Asa24]{AsayamaThesis}
    T. Asayama,
    \textit{Mordell--Weil groups over large algebraic extensions over finitely generated fields},
    PhD thesis, Tokyo Institute of Technology, Tokyo, 2024.
\bibitem[BS09]{BS}
    L. Bary-Soroker,
    \textit{On pseudo algebraically closed extensions of fields},
    J. Algebra \textbf{322} (2009), 2082--2105.
\bibitem[Bas63]{Bass}
    H. Bass,
    \textit{Big projective modules are free},
    Ill.\ J. Math.\ \textbf{7} (1963), 24--31.
\bibitem[BHP20]{BHP}
    M. Bays, B. Hart, and A. Pillay,
    \textit{Universal covers of commutative finite Morley rank groups},
    J. Inst.\ Math.\ Jussieu \textbf{19} (2020), 767--799.
\bibitem[BI08]{BI}
    F. Breuer and B.-H. Im,
    \textit{Heegner points and the rank of elliptic curves over large extensions of global fields},
    Can.\ J. Math.\ \textbf{60} (2008), 481--490.
\bibitem[CI25]{CI}
    S. Choi and B.-H. Im,
    \textit{Larsen's conjecture for elliptic curves over $\mathbb{Q}$ with analytic rank at most $1$},
    preprint, 2025. \href{https://arxiv.org/abs/2502.18761}{arXiv:2502.18761}
\bibitem[FreJ74]{FreyJ}
    G. Frey and M. Jarden,
    \textit{Approximation theory and the rank of abelian varieties over large algebraic fields},
    Proc.\ Lond.\ Math.\ Soc.\ \textbf{28} (1974) 112--128.
\bibitem[FriJ23]{FriedJ}
    M. D. Fried and M. Jarden,
    \textit{Field arithmetic}, fourth edition, revised by M. Jarden,
    Ergebnisse der Mathematik und ihrer Grenzgebiete. 3. Folge,
    A Series of Modern Surveys in Mathematics \textbf{11},
    Springer, Cham, 2023.
\bibitem[GJ78]{GJ78}
    W.-D. Geyer and M. Jarden,
    \textit{Torsion points of elliptic curves over large algebraic extensions of finitely generated fields},
    Isr.\ J. Math.\ \textbf{31} (1978) 257--297.
\bibitem[GJ06]{GJ06}
    W.-D. Geyer and M. Jarden,
    \textit{The rank of abelian varieties over large algebraic fields},
    Arch.\ Math.\ \textbf{86} (2006), 211--216.
\bibitem[GHP15]{GHP}
    R. Grizzard, P. Habegger, and L. Pottmeyer,
    \textit{Small points and free abelian groups},
    Int.\ Math.\ Res.\ Not.\ \textbf{2015} (2015), 10657--10679.
\bibitem[Hab13]{Habegger}
    P. Habegger,
    \textit{Small height and infinite nonabelian extensions},
    Duke Math.\ J. \textbf{162} (2013), 2027--2076.
\bibitem[Hos17]{Hoshi}
    Y. Hoshi,
    \textit{On the Grothendieck conjecture for affine hyperbolic curves over Kummer-faithful fields},
    Kyushu J. Math.\ \textbf{71} (2017), 1--29.
\bibitem[Im06]{Im06REC}
    B.-H. Im,
    \textit{The rank of elliptic curves with rational $2$-torsion points over large fields},
    Proc.\ Am.\ Math.\ Soc.\ \textbf{134} (2006), 1623--1630.
\bibitem[IL08]{IL08}
    B.-H. Im and M. Larsen,
    \textit{Abelian varieties over cyclic fields},
    Am.\ J. Math.\ \textbf{130} (2008), 1195--1210.
\bibitem[IL13a]{IL13ir}
    B.-H. Im and M. Larsen,
    \textit{Some applications of the Hales-Jewett theorem to field arithmetic},
    Acta Arith.\ \textbf{158} (2013), 49--59.
\bibitem[IL13b]{IL13sa}
    B.-H. Im and M. Larsen,
    \textit{Infinite rank of elliptic curves over $\mathbb{Q}^\mathrm{ab}$},
    Isr.\ J. Math.\ \textbf{198} (2013), 35--47.
\bibitem[IL21]{IL21}
    B.-H. Im and M. Larsen,
    \textit{Abelian varieties and finitely generated Galois groups},
    in: M.~Jarden and T.~Shaska (eds.) Abelian varieties and number theory,
    Contemp.\ Math.\ \textbf{767} (2021), 1--12.
\bibitem[JJ84]{JJ84}
    M. Jacobson and M. Jarden,
    \textit{On torsion of abelian varieties over large algebraic extensions of finitely generated fields},
    Mathematika \textbf{31} (1984), 110--116.
\bibitem[JJ85]{JJ85}
    M. Jacobson and M. Jarden,
    \textit{On torsion of abelian varieties over large algebraic extensions of finitely generated fields: Erratum},
    Mathematika \textbf{32} (1985), 316.
\bibitem[JJ01]{JJ01}
    M. Jacobson and M. Jarden,
    \textit{Finiteness theorems for torsion of abelian varieties over large algebraic fields},
    Acta Arith.\ \textbf{98} (2001), 15--31.
\bibitem[Jar75]{Jarden75}
    M. Jarden,
    \textit{Roots of unity over large algebraic fields},
    Math.\ Ann.\ \textbf{213} (1975), 109--127.
\bibitem[JP19]{JP19}
    M. Jarden and S. Petersen,
    \textit{Torsion of abelian varieties over large algebraic extensions of $\mathbb{Q}$},
    Nagoya Math.\ J. \textbf{234} (2019), 46--86.
\bibitem[JP22]{JP22}
    M. Jarden and S. Petersen,
    \textit{The section conjecture over large algebraic extensions of finitely generated fields},
    Math.\ Nachr.\ \textbf{295} (2022), 890--911.
\bibitem[Kob06]{Kobayashi}
    E. Kobayashi,
    \textit{A remark on the Mordell-Weil rank of elliptic curves over the maximal abelian extension of the rational number field},
    Tokyo J. Math.\ \textbf{29} (2006), 295--300.
\bibitem[Lan02]{Lang02}
    S. Lang,
    \textit{Algebra},
    revised third edition,
    Graduate Texts in Mathematics \textbf{211}, Springer-Verlag, New York, 2002.
\bibitem[Lar03]{Larsen03}
    M. Larsen,
    \textit{Rank of elliptic curves over almost separably closed fields},
    Bull.\ Lond.\ Math.\ Soc.\ \textbf{35} (2003), 817--820.
\bibitem[Lar05]{Larsen05}
    M. Larsen,
    \textit{A Mordell-Weil theorem for abelian varieties over fields generated by torsion points},
    preprint, 2005. \href{https://arxiv.org/abs/math/0503378}{arXiv:math/0503378}
\bibitem[Low23]{Lowry}
    N. Lowry,
    \textit{Abelian varieties over real closed fields},
    preprint, 2023. \href{https://arxiv.org/abs/2305.19246}{arXiv:2305.19246}
\bibitem[MT22]{MT}
    A. Minamide and S. Tsujimura,
    \textit{Anabelian group-theoretic properties of the absolute Galois groups of discrete valuation fields},
    J. Number Theory \textbf{239} (2022), 298--334.
\bibitem[Moc99]{Mochizuki99}
    S. Mochizuki,
    \textit{The local pro-$p$ anabelian geometry of curves},
    Invent.\ Math.\ \textbf{138} (1999), 319--423.
\bibitem[Moc15]{Mochizuki15}
    S. Mochizuki,
    \textit{Topics in absolute anabelian geometry III: global reconstruction algorithms},
    J. Math.\ Sci.\ Univ.\ Tokyo \textbf{22} (2015), 939--1156.
\bibitem[Moo09]{Moon09}
    H. Moon,
    \textit{On the Mordell-Weil groups of Jacobians of hyperelliptic curves over certain elementary abelian $2$-extensions},
    Kyungpook Math.\ J. \textbf{49} (2009), 419--424.
\bibitem[Moo12]{Moon12}
    H. Moon,
    \textit{On the structure of the Mordell-Weil groups of the Jacobians of curves defined by $y^n = f(x)$},
    Math.\ J. Okayama Univ.\ \textbf{54} (2012), 49--52.
\bibitem[Mur23]{Murotani}
    T. Murotani,
    \textit{Anabelian properties of infinite algebraic extensions of finite fields},
    preprint, 2023. \href{https://arxiv.org/abs/2304.13913}{arXiv:2304.13913}
\bibitem[Neu99]{Neukirch}
    J. Neukirch,
    \textit{Algebraic number theory},
    Grundlehren der mathematischen Wissenschaften \textbf{322},
    Springer-Verlag, Berlin Heidelberg, 1999.
\bibitem[Oht22]{Ohtani22}
    S. Ohtani,
    \textit{Kummer-faithful fields which are not sub-$p$-adic},
    Res.\ Number Theory \textbf{8} (2022), Paper No.~15, 7~pp.
\bibitem[Oht23]{Ohtani23}
    S. Ohtani,
    \textit{Corrigendum to: Kummer-faithful fields which are not sub-$p$-adic},
    Res.\ Number Theory \textbf{9} (2023), Paper No.~36, 7~pp.
\bibitem[OT22]{OT}
    Y. Ozeki and Y. Taguchi,
    \textit{A note on highly Kummer-faithful fields},
    Kodai Math.\ J. \textbf{45} (2022), 49--64.
\bibitem[Pet06]{Petersen}
    S. Petersen,
    \textit{On a question of Frey and Jarden about the rank of abelian varieties},
    J. Number Theory \textbf{120} (2006), 287--302.
\bibitem[RW02]{RW}
    M. Rosen and S. Wong,
    \textit{The rank of abelian varieties over infinite Galois extensions},
    J. Number Theory \textbf{92} (2002), 182--196.
\bibitem[SY12]{SY}
    F. Sairaiji and T. Yamauchi,
    \textit{The rank of Jacobian varieties over the maximal abelian extensions of number fields: towards the Frey--Jarden conjecture},
    Can.\ Math.\ Bull.\ \textbf{55} (2012), 842--849.
\bibitem[Ser68]{Serre}
    J.-P. Serre,
    \textit{Corps locaux}, deuxi\`{e}me \'{e}dition,
    Publications de l'Universit\'{e} de Nancago \textbf{VIII}, Hermann, Paris, 1968.
\bibitem[Zyw16]{Zywina16}
    D. Zywina,
    \textit{Abelian varieties over large algebraic fields with infinite torsion},
    Isr.\ J. Math.\ \textbf{211} (2016), 493--508.
\end{thebibliography}

\end{document}